\documentclass{amsart}[12pt]
\usepackage{amssymb,amsmath}
\usepackage{enumerate}
\usepackage{color}

\newcommand{\ep} {\epsilon}
\newcommand{\gm} {\gamma}

\newcommand{\ii }{\infty}

\newcommand{\ol} {\overline}

\newcommand{\dt }{\delta}
\newcommand{\lb }{\lambda}

\newcommand{\al} {\alpha}
\newcommand{\bt} {\beta}
\newcommand{\su} {\subset}
\newcommand{\mc} {\mathcal}
\newcommand{\LM} {L_0(\mc M,\tau)}
\newcommand{\PM} {\mc P(\mc M)}

\newtheorem{teo}{Theorem}[section]
\newtheorem{pro}{Proposition}[section]
\newtheorem{cor}{Corollary}[section]

\newtheorem{lm}{Lemma}[section]

\theoremstyle{definition}
\newtheorem{rem}{Remark}[section]

\title{Ergodic theorems in fully symmetric spaces \\ of $\tau-$measurable operators}
\begin{document}
\date{February 8, 2015}
\begin{abstract}
In \cite{jx}, employing the technique of noncommutative interpolation,
a maximal ergodic theorem in noncommutative $L_p-$spaces, $1<p<\ii$, was established
and, among other things, corresponding maximal ergodic inequalities and individual ergodic theorems were
derived.
In this article, we derive maximal ergodic inequalities
in noncommutative $L_p-$spaces directly from \cite{ye} and apply them to prove corresponding individual and Besicovitch
weighted ergodic theorems.
Then we extend these results to noncommutative fully symmetric Banach spaces with Fatou property and
non-trivial Boyd indices, in particular, to noncommutative Lorentz spaces $L_{p,q}$.
Norm convergence of ergodic averages in noncommutative fully symmetric Banach spaces is also studied.

\end{abstract}

\author{Vladimir Chilin}
\address{National University of Uzbekistan\\ Tashkent,  700174, Uzbekistan}
\email{chilin@ucd.uz}

\author{Semyon Litvinov}
\address{Pennsylvania State University\\ Hazleton, PA 18202, USA}
\email{snl2@psu.edu}

\maketitle

\section{Preliminaries and introduction}

Let $\mc H$ be a Hilbert space over $\mathbb C$,
$B(\mc H)$ the algebra of all bounded linear operators in $\mc H$,
$\| \cdot \|_{\ii}$ the uniform norm in $B(\mc H)$, $\Bbb I$ the identity in $B(\mc H)$. If
$\mc M\su B(\mc H)$ is a von Neumann algebra, denote by $\PM=\{ e\in \mc M: \ e=e^2=e^*\}$ the complete lattice
of all projections in $\mc M$. For every $e\in \PM$ we write $e^{\perp}=\Bbb I-e$.
If $\{ e_i\}_{i\in I}\su \PM$, the projection on the subspace
$\bigcap \limits_{i\in I}e_i(\mc H)$ is denoted by $\bigwedge \limits_{i\in I}e_i$.

A linear operator $x: \mc D_x\to \mc H$, where the domain $\mc D_x$ of $x$ is a linear subspace of $\mc H$,
 is said to be {\it affiliated with the algebra $\mc M$} if
$yx\subseteq xy$ for every $y$ from the commutant of $\mc M$.

Assume now that $\mc M$ is a semifinite von Neumann algebra equipped with a faithful
normal semifinite trace $\tau$. A densely-defined closed linear operator $x$ affiliated with $\mc M$
is called {\it $\tau$-measurable} if for each $\ep >0$
there exists such $e\in \PM$ with $\tau (e^{\perp})\leq \ep$ that $e(\mc H)\su \mc D_x$.
Let us denote by $\LM$ the set of all $\tau$-measurable operators.

It is well-known \cite{se} that if $x,y\in \LM$, then the operators $x+y$ and $xy$ are densely-defined
and preclosed. Moreover, the closures $\overline{x+y}$ (the strong sum) and $\overline{xy}$ (the strong product)
and $x^*$ are also $\tau$-measurable and, equipped with these operations, $\LM$ is a unital $*$-algebra
over $\Bbb C$.

For every subset If $X\su \LM$, the set of all self-adjoint operators in $X$ is denoted by $X^h$, whereas the set of all positive
operators in $X$ is denoted by $X^+$. The partial order $\leq$ in $L_0^h(\mc M,\tau)$ is defined by the cone $L_0^+(\mc M,\tau)$.

The topology defined in $\LM$ by the family
$$
V(\ep,\dt)=\{ x\in \LM : \|xe\|_{\ii}\leq \dt \text { \ for some } e\in \PM \text {\ with } \tau (e^{\perp})\leq \ep \}
$$
$$
\left (W(\ep,\dt)=\{ x\in \LM : \|exe\|_{\ii}\leq \dt \text { \ for some } e\in \PM \text {\ with } \tau (e^{\perp})\leq \ep \} \right ),
$$
$\ep>0, \dt>0$, of (closed) neighborhoods of zero is called the {\it measure topology} (resp.,
the {\it bilaterally measure topology}). It is said that a sequence $\{ x_n\} \su \LM$ converges to
$x\in \LM$ in measure (bilaterally in measure) if this sequence converges to $x$ in measure topology
(resp., in bilaterally measure topology). It is known \cite[Theorem 2.2]{cls} that $x_n\to x$ in measure
if and only if $x_n\to x$ bilaterally in measure. For basic properties of the measure topology in $\LM$, see \cite{ne}.

A sequence $\{ x_n\}\su \LM$ is said to converge to $x\in \LM$ {\it almost uniformly (a.u.)}
({\it bilaterally almost uniformly (b.a.u.)}) if for every $\ep>0$ there exists such $e\in \PM$
that $\tau(e^{\perp})\leq \ep$ and $\| (x-x_n)e\|_{\ii}\to 0$ (resp.,  $\| e(x-x_n)e\|_{\ii}\to 0$).
It is clear that every a.u. convergent (b.a.u. convergent) to $x$ sequence in $\LM$ converges
to $x$ in measure (resp., bilaterally in measure, hence in measure).

For a positive self-adjoint operator $x=\int_{0}^{\infty}\lambda de_{\lambda}$ affiliated with $\mathcal M$ one can define
$$
\tau(x)=\sup_{n}\tau \left (\int_{0}^{n}\lambda de_{\lambda}\right )=\int_{0}^{\infty} \lb d\tau(e_{\lambda}).
$$
If $1\leq p< \infty$, then the {\it noncommutative $L_p-$space associated with $(\mathcal M, \tau)$} is defined as
$$
L_p=( L_p(\mc M,\tau), \| \cdot \|_p) =\{ x\in \LM: \| x\|_p=(\tau (| x|^p))^{1/p}<\ii \},
$$
where $|x|=(x^*x)^{1/2}$, the absolute value of $x$ (see \cite{ye1}). Naturally, $L_{\ii}=(\mc M, \| \cdot \|_{\ii})$.
If $x_n,x\in L_p$ and $\| x-x_n\|_p\to 0$, then $x_n\to x$ in measure \cite[Theorem 3.7]{fk}. Besides, utilizing
the spectral decomposition of $x\in L_p^+$, it is possible to find a sequence $\{ x_n\}\su L_p^+\cap \mc M$
such that $0\leq x_n\leq x$ for each $n$ and $x_n \uparrow x$; in particular, $\| x_n\|_p\leq \| x\|_p$ for all $n$
and $\| x-x_n\|_p\to 0$.

Let $T: L_1\cap \mc M \to L_1\cap \mc M$ be a positive linear map that satisfies conditions of \cite{ye}:
$$
(Y) \ \ \  T(x)\leq \Bbb I \text{ \ and \ } \tau(T(x))\leq \tau(x) \ \ \forall \ x\in L_1\cap \mc M \text{ \ with\ } 0\leq x\leq \Bbb I.
$$
It is known \cite[Proposition 1]{ye} that such a $T$ admits a unique positive ultraweakly continuous linear extension
$T:\mc M\to \mc M$. In fact, $T$ contracts $\mc M$:

\begin{pro}\label{p0}
Let $T$ be the extension to $\mc M$ of a positive linear map \\ $T: L_1\cap \mc M \to L_1\cap \mc M$ satisfying condition (Y).
Then $\| T(x)\|_{\ii}\leq \|x \|_{\ii}$ for every $x\in \mc M$.
\end{pro}

\begin{proof}  Since the trace $\tau$ is semifinite, there exists a net $\{ p_{\al}\}_{\al \in \Lambda}\su \PM$, where $\Lambda$ is a base of neighborhoods of zero of the ultraweak topology ordered by inclusion,
such that $0<\tau(p_{\al})<\ii$ for erery $\al$ and $p_{\al}\to \Bbb I$ ultraweakly. Then
$T(x_{\al}) \to T(\Bbb I)$ ultraweakly. Since
$\| T(p_{\al})\|_{\ii}\leq 1$, and the unit ball of $\mc M$ is closed in ultraweak topology, we
conclude that $\| T(\Bbb I)\|_{\ii}\leq 1$. Therefore, by \cite[Corollary 2.9]{pa},
$$
\|T\|_{\mc M\to \mc M}=\|T(\Bbb I)\|_{\ii}\leq 1.
$$
\end{proof}

In \cite[Theorem 4.1]{jx}, a maximal ergodic theorem in noncommutative $L_p-$spaces, $1<p<\ii$, was established
for the class of positive linear maps $T:\mc M \to \mc M$ satisfying the condition
$$
(JX)  \ \ \  \| T(x)\|_{\ii}\leq \|x\|_ {\ii} \ \ \forall \ x\in \mc M  \text{ \ and \ }  \tau(T(x))\leq \tau(x) \ \  \forall \ x\in L_1\cap \mc M^+.
$$

\begin{rem}
Due to Proposition \ref{p0}, (JX)$\ \Leftrightarrow \ $(Y).
\end{rem}
\noindent
Besides, by \cite[Lemma 1.1]{jx}, a positive linear map $T:\mc M \to \mc M$
that satisfies (JX) uniquely extends to a positive linear contraction $T$ in $L_p$, $1< p<\ii$.

In the sequel, we shall write $T\in DS^+=DS^+(\mc M,\tau)$ to indicate that the map $T: L_1+\mc M\to L_1+\mc M$
is the unique positive linear extension of a positive linear map $T: \mc M\to \mc M$ satisfying condition (JX).
Such $T$  is often called {\it positive Dunford-Schwartz transformation} (see, for example, \cite{ye2}).

\vskip 5pt
Assume that $T \in DS^+$ and form its ergodic averages:
\begin{equation}\label{eq1}
M_n=M_n(T)=\frac 1 {n+1} \sum_{k=0}^n T^k, \ n=1,2, ... \ .
\end{equation}

The following fundamental result provides a maximal ergodic inequality  in $L_1$ for the averages (\ref{eq1}).

\begin{teo}
\label{Theorem 1}\label{t1} \cite{ye}
If $T\in DS^+$, then for every $x\in L_1^+$ and $\ep>0$, there is such $e\in \PM$ that
$$
\tau (e^{\perp})\leq \frac {\| x\|_{1}} \ep \text { \ and \ } \sup_n \| eM_n(x)e\|_{\ii} \leq \ep .
$$
\end{teo}

Here is a corollary of Theorem \ref{t1}, a noncommutative individual ergodic theorem of Yeadon:
\begin{teo}
\label{Theorem 2}\label{t2} \cite{ye}
If $T\in DS^+$, then for every $x\in L_1$ the averages $M_n(x)$ converge b.a.u.
to some $\widehat x\in L_1$.
\end{teo}

The next result, an extension of Theorem \ref{t2}, was established in \cite{jx}.

\begin{teo}[\cite{jx}, Corollary 6.4]\label{t3}
Let $T\in DS^+$, $1<p<\ii$, and $x\in L_p$. Then the averages $M_n(x)$ converge b.a.u. to some $\widehat x\in L_p$.
If $p\ge 2$, these averages converge also a.u.
\end{teo}

The proof of Theorem \ref{t3} in \cite{jx} is based on an application of
a weak type $(p,p)$ maximal inequality for the averages (\ref{eq1}), an $L_p-$version of Theorem \ref{t1}.
Note that the proof of this inequality itself relies on Theorem \ref{t1} and essentialy involves an intricate
technique of noncommutative interpolation.
Below (Theorem \ref{t4}) we provide a simple, based only on Theorem \ref{t1}, proof of such a maximal inequality.

As an application of Theorem \ref{t4}, we prove Besicovitch weighted noncommuative ergodic theorem in $L_p$, $1<p<\ii$,
(Theorem \ref{t9}), which contains Theorem \ref{t3} as a particular case. Theorem \ref{t9} is an extension of
the corresponding result for $L_1$ in \cite{cls}. Note that, in \cite{li},
Theorem \ref{t3} was derived from Theorem \ref{t1}  by utilizing the notion
of uniform equicontinuity at zero of a family of additive maps into $\LM$.

Having available Besicovitch weighted ergodic theorem for noncommutative \\ $L_p-$spaces with $1\leq  p<\ii$,
allows us to establish its validity for a wide class of noncommutative fully symmetric spaces with Fatou property.
As a consequence, we obtain an individual ergodic theorem in noncommutative Lorentz spaces $L_{p,q}$.

The last section of the article is devoted to a study of the mean ergodic ergodic theorem in noncommutative fully symmetric
spaces in the case where \\ $T\in DS(\mc M,\tau)$.

\section
{Maximal ergodic inequalities in noncommutative $L_p-$spaces}

Everywhere in this section $T\in DS^+$. Assume that a sequence of complex numbers
$\{ \bt_k\}_{k=0}^{\ii}$ is such that $|\bt_k|\leq C$ for every $k$.  Let us denote
\begin{equation}\label{eq2}
M_{\bt,n}=M_{\bt,n}(T)=\frac1{n+1}\sum_{k=0}^n\bt_kT^k.
\end{equation}

\begin{teo}\label{t4}If $1\leq p<\ii$,
then for every $x\in L_p$ and $\ep>0$ there is $e\in \PM$ such that
\begin{equation}\label{eq3}
\tau(e^{\perp})\leq 4 \left ( \frac {\| x\|_p} \ep \right )^p \text{ and } \ \sup_n \| eM_{\bt,n}(x)e\|_{\ii} \leq 48 C\ep.
\end{equation}
\end{teo}

\begin{proof} Let first $\bt_k\equiv 1$. In this case, $M_{\bt,n}=M_n$. Fix $\ep>0$. Assume that $x\in L_p^+$,
and let $x=\int_0^{\ii}\lb de_\lb$ be its spectral decomposition.
Since $\lb \ge \ep$ implies $\lb \leq \ep^{1-p}\lb^p$, we have
$$
\int_\ep^\ii \lb de_\lb \leq \ep^{1-p}\int_\ep^\ii \lb^p de_\lb \leq \ep^{1-p}x^p.
$$
Then we can write
\begin{equation}\label{eq6}
x=\int_0^\ep \lb de_\lb + \int_\ep^\ii \lb de_\lb\leq x_\ep+\ep^{1-p}x^p,
\end{equation}
where $x_\ep=\int_0^\ep \lb de_\lb$.

As $x^p \in L_1$, Theorem \ref{t1} entails that there exists $e\in \PM$ satisfying
$$
\tau(e^{\perp})\leq \frac {\| x^p\|_1}{\ep^p}=\left ( \frac {\| x\|_p}{\ep} \right )^p  \text{ and } \  \sup_n\| eM_n(x^p)e\|_{\ii}\leq \ep^p.
$$
It follows from (\ref{eq6}) that
$$
0\leq M_n(x)\leq M_n(x_{\ep})+\ep^{1-p}M_n(x^p) \ \ \text{and}
$$
$$
0\leq eM_n(x)e\leq eM_n(x_{\ep})e+\ep^{1-p}eM_n(x^p)e
$$
for every $n$.

Since $x_{\ep}\in \mc M$, the inequality
$$
\| T(x_\ep)\|_{\ii}\leq \| x_\ep \|_{\ii}\leq \ep
$$
holds, and we conclude that
$$
\sup_n\|eM_n(x)e\|_{\ii}\leq \ep+\ep=2\ep.
$$

If $x\in L_p$, then $x=(x_1-x_2)+i(x_3-x_4)$, where $x_j\in L_p^+$ and $\| x_j\|_p\leq \| x\|_p$
for every $j=1,\dots,4$. As we have shown, there exists $e_j\in \PM$ such that
\begin{equation}\label{eq11}
\tau(e_j^{\perp})\leq \left ( \frac {\| x_j\|_p} \ep \right )^p\leq \left ( \frac {\| x\|_p} \ep \right )^p, \ \
 \sup_n \| e_jM_n(x_j)e_j\|_{\ii} \leq 2\ep,
\end{equation}
$j=1,\dots,4$.

Now, let $\{ \bt_k\}_{k=0}^{\ii}\su \Bbb C$  satisfy $|\bt_k| \leq C$ for every $k$.
As $0\leq Re\bt_k+C\leq 2C$ and $0\leq Im\bt_k+C\leq 2C$,
it follows from the decomposition
\begin{equation}\label{eq4}
M_{\bt,n}=\frac 1{n+1} \sum_{k=0}^n(Re\bt_k+C)T^k+\frac i{n+1} \sum_{k=0}^n(Im\bt_k+C)T^k-C(1+i)M_n
\end{equation}
and (\ref{eq11}) that
$$
\sup_n\| e_jM_{\bt,n}(x_j)e_j\|_{\ii}\leq 6C\sup_n\| e_jM_n(x_j)e_j\|_{\ii}\leq 12C\ep, \ \  j=1,\dots ,4.
$$
Finally, letting $e=\bigwedge \limits_{j=1}^4e_j$, we arrive at (\ref{eq3}).
\end{proof}

\begin{rem}
Note that (\ref{eq11}) provides the following extension of the maximal ergodic inequality given in Theorem \ref{t1} for $p=1$: 
for every $x\in L_p^+$ and $\ep>0$ there exists $e\in \mc P(\mc M)$ such that 
$$
\tau(e^{\perp})\leq \left ( \frac {\| x\|_p} \ep \right )^p \text{ \ and \ }  \sup_n \| eM_n(x)e\|_{\ii} \leq 2\ep.
$$
\end{rem}

To refine Theorem \ref{t4} when $p\ge 2$ we turn to the fundamental result of Kadison \cite{ka}:

\begin{teo}[Kadison's inequality]\label{t5}
Let $S:\mc M\to \mc M$ be a positive linear map such that $S(\Bbb I)\leq \Bbb I$.
Then $S(x)^2\leq S(x^2)$ for every $x\in \mc M^h$.
\end{teo}

We will need the following technical lemma; see the proof of \cite[Theorem 2.7]{cls} or \cite[Theorem 3.1]{li}.

\begin{lm}\label{l0}
Let $\{ a_{mn}\}_{m,n=1}^{\ii}\su \LM$ be such that for any $n$ the sequence $\{ a_{mn}\}_{m=1}^{\ii}$ converges in measure
to some $a_n\in \LM$. Then there exists $\{ a_{m_kn}\}_{k,n=1}^{\ii}$ such that for any $n$ we have $a_{m_kn}\to a_n$ a.u. as $k\to \ii$.

\end{lm}

\begin{pro}[cf. \cite{jx}, proof of Remark 6.5]\label{p1}
If $2\leq p<\ii$ and $T\in DS^+$,
then for every $x\in L_p^h$ and $\ep>0$,
there exists $e\in \PM$ such that $\tau(e^{\perp})\leq \ep$ and
$$
\|eM_n(x)^2e\|_{\ii}\leq \| eM_n(x^2)e\|_{\ii}, \  n=1,2, \dots
$$
\end{pro}

\begin{proof}
Let $x=\int_{-\ii}^{\ii}\lb de_\lb$ be the spectral decomposition of $x\in L_p^h$, and let \\ $x_m=\int_{-m}^m\lb de_\lb$.
Then, since $x\in L_p$, we clearly have $\|x-x_m\|_p\to 0$. Besides, $\| x^2-x_m^2\|_{p/2}\to 0$, so
$\|M_n(x^2)-M_n(x_m^2)\|_{p/2}\to 0$ for every $n$, which implies that
$$
M_n(x_m^2)\to M_n(x^2) \text{ in measure}, \ n=1,2, \dots
$$
Also $\| M_n(x)-M_n(x_m)\|_p\to 0$ for every $n$, hence $M_n(x_m)\to M_n(x)$ in measure and
$$
M_n(x_m)^2\to M_n(x)^2 \text{ in measure}, \ n=1,2, \dots
$$
In view of Lemma \ref{l0},  it is possible to find a subsequence
$\{ x_{m_k}\} \su \{ x_m\}$  such that
$$
M_n(x_{m_k}^2)\to M_n(x^2) \text{ \ and \ } M_n(x_{m_k})^2\to M_n(x)^2 \text{ \  a.u.}, \ n=1,2, \dots
$$
Then one can construct such $e\in \PM$ that $\tau(e^{\perp})\leq \ep$ and
$$
\| eM_n(x_{m_k}^2)e\|_{\ii}\to \| eM_n(x^2)e\|_{\ii} \text{ \ and \ } \| eM_n(x_{m_k})^2e\|_{\ii} \to \| eM_n(x)^2e\|_{\ii}
$$
for every $n$.

Since, by Kadison's inequality, we have
$$
\| eM_n(x_{m_k})^2e\|_{\ii}\leq \| eM_n(x_{m_k}^2)e\|_{\ii}, \ k, n=1,2, \dots,
$$
the result follows.
\end{proof}

\begin{teo}\label{t6} If $2\leq p<\ii$, then for every $x\in L_p$ and $\ep>0$ there is such $e\in \PM$ that
\begin{equation}\label{eq5}
\tau(e^{\perp})\leq 6 \left ( \frac {\| x\|_p} \ep \right )^p \text{ and }
\ \sup_n \| M_{\bt,n}(x)e\|_{\ii} \leq 4\sqrt{C}(2+\sqrt C)\ep.
\end{equation}
\end{teo}
\begin{proof}
Pick $x\in L_p^h$. Since $x^2\in L_{p/2}^+$, referring to (\ref{eq11}), we can present $e_1\in \PM$ such  that
\begin{equation}\label{eq5b}
\tau(e_1^{\perp})\leq \left ( \frac {\| x^2\|_{p/2}}{\ep^2} \right )^{p/2}=\left ( \frac {\| x\|_p}{\ep} \right )^p
\text{ \ and \ } \sup_n\|e_1M_n(x^2)e_1 \|_{\ii}\leq 2\ep^2.
\end{equation}
By Proposition \ref{p1}, there is $e_2\in \PM$ such that
$$
\tau(e_2^{\perp})\leq \left ( \frac {\| x\|_p}{\ep} \right )^p \text{ \ and \ } \sup_n\|e_2M_n(x)^2e_2\|_{\ii}\leq \sup_n\| e_2M_n(x^2)e_2\|_{\ii}.
$$
Then, letting $e=e_1 \land e_2$, we obtain $\tau(e^{\perp})\leq 2\left ( \frac {\| x\|_p}{\ep} \right )^p$ and
$$
\sup_n \|M_n(x)e\|_{\ii}=\left ( \sup_n \| M_n(x)e\|_{\ii}^2 \right )^{1/2} =
$$
$$
=\left (\sup_n\|eM_n(x)^2e\|_{\ii}\right )^{1/2}\leq \left ( \sup_n \| eM_n(x^2)e\|_{\ii} \right )^{1/2}\leq \sqrt 2\ep.
$$

If $\{ \bt_k\}_{k=0}^{\ii}\su \Bbb C$, $|\bt_k|\leq C$, in accordance with the decomposition (\ref{eq4}), we denote
$$
M_{\bt,n}^{(R)}=\frac 1{n+1}\sum_{k=0}^n(Re\bt_k+C)T^k, \ \   M_{\bt,n}^{(I)}=\frac 1{n=1}\sum_{k=0}^n(Im\bt_k+C)T^k.
$$
Let $x=x_1+ix_2 \in L_p$, where $x_j\in L_p^h$ and $\| x_j\|_p\leq \| x\|_p$, $j=1,2$.
Since $x_1^2\in L_{p/2}^+$, it follows from (\ref{eq5b}) that there is $f_1\in \PM$ such that
$$
\tau(f_1^{\perp})\leq  \left ( \frac {\| x_1\|_p}{\ep}\right )^p \text
{ \ and \ } \sup_n \| f_1M_n(x_1^2)f_1\|_{\ii}\leq 2\ep^2.
$$
Therefore we have
$$
\sup_n \| f_1 M_{\bt,n}^{(R)}(x_1^2)f_1\|_{\ii}\leq 4C\ep^2 \text{ \ \ and \ \ } \sup_n \| f_1 M_{\bt,n}^{(I)}(x_1^2)f_1\|_{\ii}\leq 4C\ep^2.
$$
Since $M_{\bt,n}^{(R)}:\mc M \to \mc M$ and $M_{\bt,n}^{(I)}:\mc M \to \mc M$  are positive linear
maps satisfying $(2C)^{-1}M_{\bt,n}^{(R)}(\Bbb I)\leq \Bbb I$ and
$(2C)^{-1}M_{\bt,n}^{(I)}(\Bbb I)\leq \Bbb I$
for each $n$, applying Kadison's inequality, we obtain
$$
M_{\bt,n}^{(R)}(x_1)^2 \leq M_{\bt,n}^{(R)}(x_1^2)
$$
and
$$
M_{\bt,n}^{(I)}(x_1)^2g_{12} \leq M_{\bt,n}^{(I)}(x_1^2).
$$
This in turn entails
$$
\sup_n\| f_1M_{\bt,n}^{(R)}(x_1)^2f_1 \|_{\ii}\leq \sup_n \| f_1M_{\bt,n}^{(R)}(x_1^2)f_1\|_{\ii}
$$
and
$$
\sup_n\| f_1M_{\bt,n}^{(I)}(x_1)^2f_1 \|_{\ii}\leq \sup_n \| f_1M_{\bt,n}^{(I)}(x_1^2)f_1\|_{\ii}.
$$
Therefore
$$
\sup_n\| M_{\bt,n}^{(R)}(x_1)f_1 \|_{\ii}^2 = \sup_n\|f_1 M_{\bt,n}^{(R)}(x_1)^2f_1 \|_{\ii}\leq \sup_n\|f_1 M_{\bt,n}^{(R)}(x_1^2) f_1 \|_{\ii} \leq 4 C \ep^2,
$$
and similarly
$$
\sup_n\| M_{\bt,n}^{(I)}(x_1)f_1 \|_{\ii}^2  \leq 4 C \ep^2.
$$
 Then, letting $g_1=e \land f_1$, we derive  $\tau(g^{\perp}_1)\leq 3 \left ( \frac {\| x\|_p}{\ep} \right )^p$ and

$$
\sup_n\| M_{\bt,n}(x_1)g_1\|_{\ii}\leq  2\sqrt{C}(2+\sqrt C)\ep.
$$
Similarly, one can find $g_2\in \PM$ with $\tau (g_2^{\perp})\leq 3(\| x_2\|_p/\ep)^p$ such that
$$
\sup_n\| M_{\bt,n}(x_2)g_2\|_{\ii}\leq 2\sqrt{C}(2+\sqrt C)\ep.
$$
Finally, we conclude that $e=g_1\land g_2\in \PM$ satisfies (\ref{eq5}).
\end{proof}

\begin{rem}
Beginning of the proof of Theorem  \ref{t6} contains the following maximal ergodic inequality for the 
ergodic averages (\ref{eq1}): if $2\leq p<\ii$, given $x\in L_p^h$
and $\ep>0$, there exists $e\in \mc P(\mc M)$ such that 
$$
\tau(e^{\perp})\leq 2\left ( \frac {\| x\|_p} \ep \right )^p \text{ \ and \ }  \sup_n \| eM_n(x)e\|_{\ii} \leq \sqrt 2\ep.
$$
\end{rem}

\section
{Besicovitch weighted ergodic theorem in noncommutative $L_p-$spaces}

In this section, using maximal ergodic inequalities given in Theorems \ref{t4} and \ref{t6}, we prove Besicovitch
weighted ergodic theorem in noncommutative $L_p-$spaces, $1<p<\ii$. As was already mentioned, this extends
the corresponding result for $p=1$ from \cite{cls}. Everywhere in this section $T\in DS^+$.

We will need the following technical lemma.

\begin{lm}[see \cite{cl}, Lemma 1.6]\label{l1}  Let $X$ be a linear space, and let \\ $S_n:X\to \LM$ be a sequence of additive maps.
Assume that $x\in X$ is such that for
every $\ep>0$ there exists a sequence $\{ x_k\} \su X$ and a projection $e\in \PM$ satisfying the following conditions:

(i) the sequence $\{ S_n(x+x_k)\}$ converges a.u. (b.a.u.) as $n\to \ii$ for each $k$;

(ii) $\tau(e^{\perp})\leq \ep$;

(iii) $\sup_n \| S_n(x_k)e\|_{\ii} \to 0$ (resp., $\sup_n \| eS_n(x_k)e\|_{\ii} \to 0$) as $k\to \ii$.

\noindent
Then the sequence $\{ S_n(x)\}$ also converges a.u. (resp., b.a.u.)
\end{lm}

Using Theorems \ref{t4} and \ref{t6}, we obtain a corollary:

\begin{cor}\label{c1} Let $1\leq p<\ii$ ($2\leq p<\ii$). Then the set
$$
\{ x\in L_p: \{ M_{\bt,n}(x)\} \text{ converges b.a.u.} \}
$$
$$
\left (\text{resp.,\ } \{ x\in L_p: \{ M_{\bt,n}(x)\} \text{ converges a.u.} \}\right )
$$
is closed in $L_p$.
\end{cor}

\begin{proof} Denote $C=\{ x\in L_p: \{ M_{\bt,n}(x)\} \text{ converges b.a.u.} \}$.
Fix $\ep>0$. Theorem \ref{t4} implies that for every given $k\in \Bbb N$ there is such $\gm_k>0$
that for every $x\in L_p$ with $\| x\|_p<\gm_k$ it is possible to find $e_{k,x}\in \PM$ for which
$$
\tau(e_{k,x}^{\perp})\leq \frac \ep {2^k} \text { \ and \ } \sup_n\| e_{k,x}M_{\bt,n}(x)e_{k,x}\|_{\ii}\leq \frac 1k \ .
$$

Pick $x\in \ol C$, the closure of $C$ in $L_p$. Given $k$, let $y_k\in C$ satisfy $\| y_k-x\|_p<\gm_k$.
Denoting $y_k-x=x_k$, choose a sequence $\{ e_k\}\su \PM$ to be such that
$$
\tau(e_k^{\perp})\leq \frac \ep {2^k} \text { \ and \ } \sup_n\|e_kM_{\bt,n}(x_k)e_k\|_{\ii} \leq \frac 1k, \ k=1,2, ... \ .
$$
Then we have $x+x_k=y_k\in C$ for every $k$. Also, letting $e=\bigwedge \limits_{k\ge1} e_k$, we have
$$
\tau(e^{\perp})\leq \ep \text { \ and \ } \sup_n\| eM_{\bt,n}(x_k)e\|_{\ii} \leq \frac 1k \ .
$$
Consequently, Lemma \ref{l1} yields $x\in C$.

Analogously, applying Theorem \ref{t6} instead of Theorem \ref{t4}, we obtain the remaining part of the statement.
\end{proof}

Corollary \ref{c1}, in the particular case where $\bt_k\equiv1$, allows us to present a new, direct proof of Theorem \ref{t3}.

\begin {proof}
Assume first that $p\ge 2$. Since the map $T$ generates a contraction in the real Hilbert space
$(L_2^h, (\cdot, \cdot)_{\tau})$ \cite[Proposition 1]{ye}, where $(x,y)_{\tau}=\tau(xy)$,
$x,y\in L_2^h$, it is easy to verify that the set
$$
\mc H_0=\{ x\in L_2^h:T(x)=x\}+\{ x-T(x):x\in L_2^h\}
$$
is dense in $(L_2^h, \| \cdot \|_2)$ (see, for example \cite[Ch.VIII, \S 5]{ds}). Therefore,
because the set $L_2^h\cap \mc M$ is dense in $L_p^h$ and $T$ contracts $L_p^h$, we conclude that the set
$$
\mc H_1=\{ x\in L_2^h:T(x)=x\}+\{ x-T(x):x\in L^2_h\cap \mc M\}
$$
is also dense in $(L_2^h, \| \cdot \|_2)$. Besides, if $y=x-T(x)$, $x\in L_2^h\cap \mc M$, then the sequence
$M_n(y)=(n+1)^{-1}(x-T^{n+1}(x))$ converges to zero with respect to the norm $\| \cdot \|_{\ii}$,
hence a.u. Therefore $\mc H_1+i\mc H_1$ is a dense in $L_2$
subset on which the averages $M_n$ converge a.u. This, by
Corollary \ref{c1}, implies that $\{ M_n(x)\}$ converges a.u. for all $x\in L_2$.
Further, since the set $L_p\cap L_2$ is dense in $L_p$,
Corollary \ref{c1} implies that the sequence $\{ M_n(x)\}$ converges a.u.  for each $x\in L_p$ (to some $\widehat x\in \LM$).
Then $\{ M_n(x)\}$ converges to $\widehat x$ in measure. Since $M_n(x)\in L_p$ and $\| M_n(x)\|_p\leq 1$, $n=1,2, \dots$,
by Theorem 1.2 in \cite{cls}, $\widehat x\in L_p$.

Let now $1<p<\ii$. By the first part of the proof, the sequence $\{ M_n(x)\}$ converges b.a.u. for all $x\in L_2$.
But $L_p\cap L_2$ is dense in $L_p$, and Corollary \ref{c1} entails b.a.u. convergence of the averages
$M_n(x)$ for all $x\in L_p$. Remembering that b.a.u. convergence
implies convergence in measure (see Section 1), we conclude, as before, that $M_n(x)\to \widehat x\in L_p$ b.a.u..
\end{proof}

Let $\Bbb C_1=\{z\in \Bbb C: |z|=1\}$ be the unit circle in $\mathbb C$. A function $P : \mathbb Z \to \mathbb C$ is said to be a {\it trigonometric polynomial} if
$P(k)=\sum_{j=1}^{s} z_j\lb_j^k$, $k\in \mathbb Z$, for some $s\in \mathbb N$, $\{ z_j \}_1^s \subset \mathbb C$, and $\{ \lb_j \}_1^s \subset \mathbb C_1$.
A sequence $\{ \beta_k \}_{k=0}^{\ii} \subset \Bbb C$ is called a {\it bounded Besicovitch sequence} if

(i) $| \beta_k | \leq C < \ii$ for all $k$;

(ii) for every $\ep >0$ there exists a trigonometric polynomial $P$ such that
$$
\limsup_n \frac 1{n+1} \sum_{k=0}^n | \beta_k - P(k) | < \ep .
$$

Assume now that $\mc M$ has a separable predual. The reason for this assumption is that our argument essentially relies
on \cite[Theorem 1.22.13]{sa}.

Since $L_1\cap \mc M\su L_2$, using Theorem \ref{t3} for $p=2$ (or \cite[Theorem 3.1]{col}) and repeating steps
of the proof of \cite[Lemma 4.2]{cls}, we arrive at the following.

\begin{pro}\label{p2} For any trigonometric polynomial $P$ and $x\in L_1\cap \mc M$, the averages
$$
\frac 1{n+1} \sum_{k=0}^nP(k)T^k(x)
$$
converge a.u.
\end{pro}

Next, it is easy to verify the following (see the proof of \cite[Theorem 4.4]{cls}).

\begin{pro}\label{p3}
If $\{ \bt_k\}$ is a bounded Besicovitch sequence, then the averages (\ref{eq2})
converge a.u. for every $x\in L_1\cap \mc M$.
\end{pro}

Here is an extension of \cite[Theorem 4.6]{cls} to $L_p-$spaces, $1<p<\ii$.
\begin{teo}\label{t9}
Assume that $\mc M$ has a separable predual.
Let $1< p<\ii$, and let $\{ \bt_k\}$ be a bounded Besicovitch sequence. Then for every $x\in L_p$
the averages (\ref{eq2}) converge b.a.u. to some $\widehat x\in L_p$. If $p\ge 2$, these averages converge a.u.
\end{teo}
\begin{proof}
In view of Proposition \ref{p3} and Corollary \ref{c1}, we only need to recall that
the set $L_1\cap \mc M$ is dense in $L_p$. The inclusion $\widehat x\in L_p$ follows as in the proof of Theorem \ref{t3}.
\end{proof}

\section{Individual ergodic theorems in noncommutative fully symmetric spaces}

Let $x\in \LM$, and let $\{ e_{\lb}\}_{\lb\ge 0}$ be the spectral family of projections for the absolute value $| x|$ of $x$.
If $t>0$, then the {\it $t$-th generalized singular number of $x$} (see \cite{fk}) is defined as
$$
\mu_t(x)=\inf\{\lb>0: \tau(e_{\lb}^{\perp})\leq t\}.
$$

A Banach space $(E, \| \cdot \|_E)\su \LM$ is called {\it fully symmetric} if the conditions
$$
x\in E, \ y\in \LM, \ \int \limits_0^s\mu_t(y)dt\leq  \int \limits_0^s\mu_t(x)dt \text{ \ for all\ }s>0
$$
imply that $y\in E$ and $\| y\|_E\leq \| x\|_E$. It is known \cite{ddp} that if $(E, \| \cdot \|_E)$ is a fully symmetric space,
$x_n,x\in E$, and $\|x- x_n\|_E\to 0$, then $x_n\to x$ in measure. A fully symmetric space $(E, \| \cdot \|_E)$ is said to possess
{\it Fatou property} if the conditions
$$
x_{\al}\in E^+, \ \ x_{\al}\leq x_{\bt} \text{ \ for } \al \leq \bt, \text{\ and\ }\sup_{\al} \| x_{\al}\|_E<\ii
$$
imply that there exists $x=\sup \limits_{\al}x_{\al}\in E$ and $\| x\|_E=\sup \limits_{\al} \| x_{\al}\|_E$.
The space $(E, \| \cdot \|_E)$ is said to have {\it order continuous norm} if $\| x_{\al}\|_E\downarrow 0$ whenever $x_{\al}\in E$ and $x_{\al}\downarrow 0$.

Let $L_0(0,\ii)$ be the linear space of all (equivalence classes of) almost everywhere finite complex-valued Lebesgue measurable
functions on the interval $(0,\ii)$. We identify $L_{\ii}(0,\ii)$ with the commutative von Neumann algebra acting
on the Hilbert space $L_2(0,\ii)$ via
multiplication by the elements from $L_{\ii}(0,\ii)$ with the trace given by the integration
with respect to Lebesgue measure. A Banach space $E\su L_0(0,\ii)$ is called {\it fully symmetric Banach space on $(0,\ii)$}
if the condition above holds with respect to the von Neumann algebra $L_{\ii}(0,\ii)$.

Let $E=(E(0,\ii), \| \cdot \|_E)$ be a fully symmetric function space.
For each $s>0$ let $D_s: E(0,\infty) \to E(0,\infty)$ be the bounded linear operator given by \\ $D_s(f)(t) = f(t/s), \ t > 0$.
The {\it Boyd indices} $p_E$ and $q_E$ are defined as
$$
p_E=\lim\limits_{s\to\infty}\frac{\log s}{\log \|D_s\|_E}, \ \ q_E=\lim\limits_{s \to +0}\frac{\log s}{\log \|D_{s}\|_E}.
$$
It is known that $1\leq p_E\leq q_E\leq \ii$ \cite[II, Ch.2, Proposition 2.b.2]{lt}.
A fully symmetric function space is said to have {\it non-trivial Boyd indices} if $1<p_E$ and $q_E<\ii$.
For example, the spaces $L_p(0,\ii)$, $1< p<\ii$, have non-trivial Boyd indices:
$$
p_{L_p(0,\infty)} = q_{L_p(0,\infty)} = p
$$
\cite[Ch.4, \S 4, Theorem 4.3]{bs}.

\vskip 5 pt
If $E(0,\ii)$ is a fully symmetric function space, define
$$
E(\mc M)=E(\mc M, \tau)=\{ x\in \LM: \mu_t(x)\in E\}
$$
and set
$$
\| x\|_{E(\mc M)}=\| \mu_t(x)\|_E,  \ x\in E(\mc M).
$$
It is shown in \cite{ddp} that $(E(\mc M), \| \cdot \|_{E(\mc M)})$ is a fully symmetric space.
If $1\leq p<\ii$ and $E=L_p(0,\ii)$, the space $(E(\mc M), \| \cdot \|_{E(\mc M)})$ coincides with the noncommutative
$L_p-$space $(L_p(\mc M, \tau), \| \cdot \|_p)$ because
$$
\| x\|_p=\left (\int \limits_0^{\ii}\mu_t^p(x)dt\right )^{1/p}=\| x\|_{E(\mc M)}
$$
\cite[Proposition 2.4]{ye1}.

It was shown in \cite[Proposition 2.2]{cs} that if $\mc M$ is non-atomic, then every noncommutative fully symmetric $(E, \| \cdot \|_E)\su \LM$ is of the form $(E(\mc M), \| \cdot \|_{E(\mc M)})$ for a suitable fully symmetric  function space $E(0,\infty)$.

Let $L_{p,q}(0,\ii)$, $1\leq p,q<\ii$, be the classical function Lorentz space, that is, the space of all such functions
$f\in L_0(0,\ii)$ that
$$
\| f\|_{p,q}=\left (\int \limits_0^{\ii}(t^{1/p}\mu_t(f))^q\frac {dt}t \right )^{1/q}<\ii.
$$
It is known that for $q\leq p$ the space $(L_{p,q}(0,\ii), \| \cdot \|_{p,q})$ is a
fully symmetric function space with Fatou property and order continuous norm. In addition,
$L_{p,p}=L_p$. In the case $1<p<q$, the function $\| \cdot \|_{p,q}$ is a quasi-norm on $L_{p,q}(0,\ii)$, but
there exists a norm $|| \cdot ||_{(p,q)}$ on $L_{p,q}(0,\ii)$ that is equivalent to the norm $\| \cdot \|_{p,q}$
and such that  $(L_{p,q}(0,\ii), || \cdot ||_{(p,q)})$ is a fully symmetric function space with Fatou property
and order continuous norm \cite[Ch.4, \S4]{bs}.  In addition, if $1 \leq q \leq p < \ii$ ($1 < p  < \ii, \ 1 \leq q < \ii$), then
$$
p_{(L_{p,q}(0,\infty), \|\cdot\|_{p,q})} = q_{(L_{p,q}(0,\infty), \|\cdot\|_{p,q})} = p
$$  \cite[Ch.4, \S 4, Theorem 4.3]{bs}
(resp.,
$$
p_{(L_{p,q}(0,\infty), \|\cdot\|_{(p,q)})} = q_{(L_{p,q}(0,\infty), \|\cdot\|_{(p,q)})} = p
$$
\cite[Ch.4, \S 4, Theorem 4.5]{bs}).

\vskip 5 pt
Using function Lorentz space $(L_{p,q}(0,\ii), \| \cdot \|_{p,q})$ ($(L_{p,q}(0,\ii), \| \cdot \|_{(p,q)})$), one can define
{\it noncommutative Lorentz space}
$$
L_{p,q}(\mc M,\tau)=\left \{ x\in \LM: \| x\|_{p,q}=\left (\int \limits_0^{\ii}(t^{1/p}
\mu_t(x))^q\frac {dt}t\right )^{1/q}<\ii \right  \}
$$
that is fully symmetric with respect to the norm $\| \cdot \|_{p,q}$ for $1\leq q\leq p$ (resp., with respect
to the norm $|| \cdot ||_{(p,q)}$ for $q>p>1$).
In addition, the norm $\| \cdot \|_{p,q}$ (resp., $|| \cdot ||_{(p,q)}$) is order continuous \cite[Proposition 3.6]{ddp1}
and satisfies Fatou property \cite[Theorem 4.1]{ddst}. These spaces were first introduced in the paper \cite{ko}.

Following \cite{kps}, a {\it Banach couple} $(X,Y)$ is a pair of Banach spaces,
$(X,\| \cdot \|_X)$ and $(Y,\| \cdot \|_Y)$, which are algebraically and topologically embedded in a Housdorff topological space.
With any Banach couple $(X,Y)$ the following Banach spaces are associated:

(i) the space $ X\cap Y $ equipped with the norm
$$
\|x\|_{X\cap Y}=\max\{\|x\|_X,\|x\|_Y\}, \ x\in X\cap Y;
$$

(ii) the space $X+Y$  equipped with the norm
$$
\|x\|_{X+Y}=\inf\{\|y\|_X+\|z\|_Y: x=y+z,  y\in X, z \in Y \}, \ x\in X+Y.
$$

Let $(X, Y)$ be a Banach couple. A linear map $T: X+Y \rightarrow X+Y$  is called a
{\it bounded operator for the couple $(X, Y)$} if both $T:X\to X$  and $T:Y\to Y$  are bounded.
Denote by $\mathcal{B}(X,Y)$ the linear space of all bounded
linear operators for the couple $(X, Y)$. Equipped with the norm
$$
\|T\|_{\mathcal{B}(X,Y)}=\max\{ \|T\|_{X \rightarrow X},\|T\|_{Y \rightarrow Y}\},
$$
this space is a Banach space.
A Banach space $Z$ is said to be {\it intermediate for
a Banach couple $(X,Y)$} if
$$
X \cap Y \subset Z \subset X+Y
$$
with continuous inclusions. If $Z$ is intermediate for a Banach couple $(X,Y)$, then it is called
an {\it interpolation space for $(X,Y)$} if every bounded linear operator for the couple
 $(X, Y)$  acts boundedly from $Z$ to $Z$.

If $Z$ is an interpolation space for a Banach couple
$(X, Y)$, then there exists a constant $C>0$ such that $\|T\|_{Z \rightarrow Z}\leq
C \|T\|_{\mathcal{B}(X,Y)}$ for all $T \in \mathcal{B}(X,Y)$.
An interpolation space $Z$ for a Banach couple
$(X,Y)$  is called an {\it exact interpolation space}  if
$\|T\|_{Z \rightarrow Z} \leq \|T\|_{\mathcal{B}(X,Y)}$ for all $T\in {\mathcal{B}(X,Y)}.$

Every fully symmetric function space $E=E(0,\infty)$ is an exact interpolation space for the Banach couple
$(L_1(0,\infty),L_{\ii}(0,\infty))$ \cite[Ch.II, \S 4, Theorem 4.3]{kps}.

\vskip 5 pt
We need  the following noncommutative interpolation result for
the spaces $E(\mc M)$.

\begin{teo}\cite[Theorem 3.4]{ddp}\label{t40}
Let $E,E_1,E_2$ be fully symmetric function spaces on $(0,\infty)$. Let $\mathcal{M}$ be a von Neumann
algebra with a faithful semifinite normal trace. If $(E_1, E_2)$ is a Banach couple and $E$ is an exact interpolation
space for $(E_1, E_2)$, then  $E(\mathcal{M})$ is an exact
interpolation space for the Banach couple $(E_1(\mathcal{M}), E_2(\mathcal{M}))$.
\end{teo}

It follows now from \cite[Ch.II, Theorem 4.3]{kps} and Theorem \ref{t40} that every non-commutative fully symmetric space $E(\mathcal{M})$, where $E=E(0,\infty)$ is a fully symmetric function space, is an exact
interpolation space for the Banach couple $(L_1(\mc M), \mc M)$.

\vskip 5pt
Let $T\in DS^+(\mc M,\tau)$. Let $E(0,\ii)$ be
a fully symmetric function space. Since the noncommutative fully symmetric space $E(\mathcal{M})$ is
an exact interpolation space for the Banach couple $(L_1(\mathcal{M}, \tau), \mathcal{M})$, we conclude that
$T(E(\mathcal{M})) \subset E(\mathcal{M})$ and $T$ is a positive  linear contraction  on
$(E(\mathcal{M}), \|\cdot\|_{E(\mathcal{M})})$. Thus
$$
M_n(x)=\frac {1} {n+1} \sum_{k=0}^{n}T^k(x) \in E(\mathcal{M})
$$
for each $x \in E(\mathcal{M})$ and all $n$.
Besides, the inequalities
$$
\|T(x)\|_1 \leq \|x\|_1, \ x \in L_1, \ \ \|T(x)\|_{\ii} \leq \|x\|_{\ii}, \ x \in  \mathcal{M}
$$
imply that
$$
\sup\limits_{n\geq 1}\|M_n\|_{L_1 \rightarrow L_1} \leq 1 \text{ \ and \ }  \sup\limits_{n\geq 1}\|M_n\|_{\mathcal{M} \rightarrow \mathcal{M}} \leq 1.
$$
Since the noncommutative fully symmetric space $E(\mathcal{M})$ is an exact interpolation space for the Banach couple $(L_1(\mathcal{M}, \tau), \mathcal{M})$, we have
\begin{equation}\label{eq42}
\sup\limits_{n\geq 1}\|M_n\|_{E(\mathcal{M}) \rightarrow E(\mathcal{M})} \leq 1.
\end{equation}

Now, let $\{ \bt_k\}_{k=0}^{\ii}\su \Bbb C$  satisfy $|\bt_k| \leq C$, $k=1,2,\dots$. \
As $0\leq Re\bt_k+C\leq 2C$ and $0\leq Im\bt_k+C\leq 2C$, it follows from (\ref{eq4}) that
$$
\sup\limits_{n\geq 1}\|M_{\bt,n}\|_{L_1 \rightarrow L_1} \leq 6C \text{ \ and \ }  \sup\limits_{n\geq 1}\|M_{\bt,n}\|_{\mathcal{M} \rightarrow \mathcal{M}} \leq 6C.
$$
Since the noncommutative fully symmetric space $E(\mathcal{M})$ is an exact interpolation space for the Banach couple $(L_1(\mathcal{M}, \tau), \mathcal{M})$, we obtain
\begin{equation}\label{eq42a}
\sup\limits_{n\geq 1}\|M_{\bt,n}\|_{E(\mathcal{M}) \rightarrow E(\mathcal{M})} \leq 6C.
\end{equation}

The following theorem is a version of Theorem \ref{t3} for
noncommutative fully symmetric Banach spaces with non-trivial Boyd indices.

\begin{teo}\label{t42} Let $E(0,\ii)$ be a fully symmetric function space with Fatou property and non-trivial
Boyd indices. If $T\in DS^+(\mc M, \tau)$, then for any given $x\in E(\mc M,\tau)$
the averages $M_n(x)$ converge b.a.u. to some  $\widehat{x} \in E(\mc M,\tau)$.
If $p_{E(0,\ii)} > 2$, these averages converge a.u.
\end{teo}

\begin{proof}
Since $E(0,\ii)$ has non-trivial Boyd indices, according to  \cite[II, Ch.2, Proposition 2.b.3]{lt},
there exist such $1<p, q<\ii$ that the space $E(0,\ii)$ is intermediate for the
Banach couple $(L_p(0,\ii), L_q(0,\ii))$. Since
$$
(L_p+L_q)(\mc M,\tau)=L_p(\mc M, \tau)+L_q(\mc M, \tau)
$$
(see \cite[Proposition 3.1]{ddp}), we have
$$
 E(\mc M,\tau) \su  L_p(\mc M,\tau) +  L_q(\mc M,\tau).
$$
Then $x = x_1 +x_2$, where $x_1\in L_p(\mc M,\tau)$, $x_2\in L_q(\mc M,\tau)$, and, by
Theorem \ref{t3}, there exist such $\widehat{x}_1 \in L_p(\mc M,\tau)$ and $\widehat{x}_1 \in L_q(\mc M,\tau)$
that $M_n(x_j)$  converge b.a.u. to $\widehat{x}_j$, $j=1,2$. Therefore
$$
M_n(x)\to \widehat x=\widehat{x}_1 +\widehat{x}_2 \in L_p(\mc M,\tau) +L_q(\mc M,\tau) \su \LM
$$
b.a.u., hence $M_n(x) \to \widehat x$ in measure.
Since $E(\mathcal{M})$ satisfies Fatou property, the unit ball of $E(\mathcal{M})$
is closed in the measure topology \cite[Theorem 4.1]{ddst}, and (\ref{eq42}) implies that $\widehat x \in E(\mathcal{M})$.

If $p_{E(0,\ii)} > 2$, then the numbers $p$ and $q$ can be chosen such that $2<p, q<\ii$.
Utilizing Theorem \ref{t3} and repeating the argument above, we conclude
that the averages $M_n(x)$ converge to $\widehat x$ a.u.
\end{proof}

Following the proof of Theorem \ref{t42}, we obtain its extended version:

\begin{teo}\label{t46}
Let $E(0,\ii)$ be a fully symmetric function space with Fatou property. If $T\in DS^+$ and $x\in E(\mc M,\tau)$ is such that
$x=x_1+ \dots +x_{n(x)}$, where $x_i\in L_{p_j(x)}(\mc M,\tau)$, $p_j(x)\ge 1$, $j=1,\dots , n(x)$, then the averages
$M_n(x)$ converge b.a.u. to some $\widehat x\in E(\mc M,\tau)$. If $p_j(x)\ge 2$ for all $j=1, \dots , n(x)$, these averages
converge a.u.
\end{teo}

Since any function Lorentz space $E=L_{p,q}(0,\ii)$ with $1<p<\ii$ and $1\leq q<\ii$ has non-trivial Boyd indices $p_E=q_E=p$,
we have the following corollary of Theorem \ref{t42}.

\begin{teo}\label{t45}
Let $1<p<\ii$ and $1\leq q<\ii$. Then, given $x\in L_{p,q}(\mc M,\tau)$,
the averages $M_n(x)$ converge b.a.u. to some  $\widehat{x} \in L_{p,q}(\mc M,\tau)$.
If $p> 2$, these averages converge a.u.
\end{teo}

\begin{rem}
If $1\leq q \leq p$, then $L_{p,q}(\mc M,\tau) \su L_{p,p}(\mc M,\tau) = L_{p}(\mc M,\tau)$ 
(see \cite{ko} and \cite[Lemma 1.6]{yb}). Then it follows directly from Theorem \ref{t3}
along with the ending of the proof ofthe first part of Theorem \ref{t42} that for every $x\in L_{p,q}(\mc M,\tau)$
the averages $M_n(x)$ converge to some $\widehat x\in L_{p,q}(\mc M,\tau)$  b.a.u. (a.u. for $p\ge 2$).
\end{rem}

The next theorem is a version of  Besicovitch weighted
ergodic theorem for a noncommutative fully symmetric space $E(\mc M,\tau)$.
\begin{teo}\label{t44a}
Assume that $\mc M$ has a separable predual.  Let $E(0,\ii)$ be a fully
symmetric function space with Fatou property and non-trivial
Boyd indices. Let $\{ \bt_k\}$ be a bounded Besicovitch sequence. If $T\in DS^+(\mc M, \tau)$, then for any given $x\in E(\mc M,\tau)$
the averages $M_{\bt,n}(x)$ converge b.a.u. to some  $\widehat{x} \in E(\mc M,\tau)$.
If $p_{E(0,\ii)} > 2$, these averages converge a.u.
\end{teo}
\noindent
Proof of Theorem \ref{t44a} uses Theorem \ref{t9} and the inequality (\ref{eq42a})  and is analogous to the proof of Theorem \ref{t42}.

 Immediately From Theorem \ref{t44a} we obtain the following individual ergodic theorem 
for Lorentz spaces $L_{p,q}(\mc M,\tau)$ (cf. Theorem \ref{t45}).
\begin{teo}\label{t44}
Let $\mc M$ have a separable predual. If $1 < p <\ii$ and $1 \leq q <\ii$, then for any $x \in L_{p,q}(\mc M,\tau)$
the averages $M_{\bt,n}$
converge b.a.u. to some  $\widehat x \in L_{p,q}(\mc M,\tau)$. If $p >2$, these averages converge a.u.
\end{teo}

\begin{rem}
If $1\leq q \leq p$, then $L_{p,q}(\mc M,\tau) \subset L_{p}(\mc M,\tau)$, and
it follows directly from Theorem \ref{t9}
along with the ending of the proof of the first part of Theorem \ref{t42} that for every $x\in L_{p,q}(\mc M,\tau)$
the averages $M_{\bt,n}$ converge to some $\widehat x\in L_{p,q}(\mc M,\tau)$  b.a.u. (a.u. for $p\ge 2$).
\end{rem}

\section
{Mean ergodic theorems in noncommutative fully symmetric spaces}
Let $\mc M$ be a von Neumann algebra with a faithful normal semifinite
trace $\tau $. In \cite{ye1} the following mean ergodic theorem for noncommutative fully symmetric spaces was proven.

\begin{teo}\label{t51}
Let $E(\mc M)$ be a noncommutative fully symmetric space such that

(i) $L_1\cap \mc M$ is dense in $E(\mc M)$;

(ii) $\|e_n\|_{E(\mc M)} \to 0$ for any sequence of projections $\{e_n\} \su L_1\cap \mc M$ with $e_n \downarrow 0$;

(iii) $\|e_n\|_{E(\mc M)}/\tau(e_n) \rightarrow 0$ for any increasing sequence of projections \\ $\{e_n\} \su  L_1\cap \mc M$
with $\tau(e_n) \to \ii$.

\noindent
Then, given $x \in E(\mc M)$ and $T \in DS^+(\mc M, \tau)$, there exists
$\widehat{x} \in E(\mc M)$ such that $\| \widehat x-M_n(x)\|_{E(\mc M)}\to 0$.
\end{teo}

It is clear that any noncommutative fully symmetric space $(E(\mc M), \|\cdot\|_{E(\mc M)})$ with order continuous norm
satisfies conditions (i) and (ii) of Theorem \ref{t51}.
Besides, in the case of noncommutative Lorentz space $L_{p,q}(\mc M,\tau)$, the inequality $p >1$
together with
$$
\|e\|_{p,q} = \left (\frac pq\right )^{1/q} \tau(e)^{1/p}, \ \ e \in L_1\cap \PM
$$
imply that condition (iii) is also satisfied. Therefore Theorem \ref{t51} entails the following.

\begin{cor}\label{c51}
Let $1<p<\ii$, $1\leq q<\ii$, $T \in DS^+$, and $x \in L_{p,q}(\mc M,\tau)$. Then there exists $\widehat{x} \in L_{p,q}(\mc M,\tau)$
such that $\|\widehat x - M_n(x)\|_{p,q}\to 0$.
\end{cor}

The next theorem asserts convergence in the norm $\|\cdot\|_{E(\mc M)}$ of the averages $M_n(x)$ for any
noncommutative fully symmetric space $(E(\mc M), \|\cdot\|_{E(\mc M)})$ with order continuous norm,
under the assumption that $\tau(\Bbb I) < \ii$.

\begin{teo}\label{t52}
Let $\tau$  be finite, and let $E(\mc M,\tau)$ be a noncommutative fully symmetric space with order continuous norm.
Then for any $x \in E(\mc M)$ and $T \in DS^+$ there exists $\widehat{x} \in E(\mc M)$ such that
$\|\widehat x-M_n(x) \|_{E(\mc M)}\to 0$.
\end{teo}

\begin{proof}
Since the trace $\tau$ is finite, we have $\mc M \su E(\mc M, \tau)$. As the norm $\|\cdot\|_{E(\mc M)}$ is order continuous,
applying spectral theorem for selfadjoint operators in $E(\mc M, \tau)$, we conclude that $\mc M$ is
dense in $(E(\mc M, \tau), \|\cdot\|_{E(\mc M)})$. Therefore $\mc M^+$  is a fundamental subset of
$(E(\mc M, \tau), \|\cdot\|_{E(\mc M)})$, that is, the linear span of
$\mc M^+$ is dense in  $(E(\mc M, \tau), \|\cdot\|_{E(\mc M)})$.

Show that the sequence $\{ M_n(x)\}$ is relatively weakly sequentially compact for every $x\in \mc M^+$.
Without loss of generality, assume that $0 \leq x \leq \Bbb I$. Since $T \in DS^+$,
we have $0 \leq M_n(x) \leq M_n(\Bbb I) \leq \Bbb I$ for any $n$. By \cite[Proposition 4.3]{dps}, given \\ $y \in E^+(\mc M, \tau)$, the set $\{a \in E(\mc M, \tau): 0 \leq a \leq y\}$ is weakly compact  in \\ $(E(\mc M, \tau), \|\cdot\|_{E(\mc M)})$, which implies that the sequence $\{M_n(x)\}$ is relatively weakly sequentially compact in $(E(\mc M, \tau), \|\cdot\|_{E(\mc M)})$.

Since $\sup\limits_{n\geq 1}\|M_n\|_{E(\mathcal{M}) \rightarrow E(\mathcal{M})} \leq 1$ (see (\ref{eq42})) and
$$
0 \leq \left \| \frac {T^n(x)}n \right \|_{E(\mc M)} \leq \frac {\|x\|_{E(\mc M)}} n \to 0
$$
whenever $x \in \mc M^+$, the result follows by Corollary 3 in \cite[Ch.VIII, \S 5]{ds}.

\end{proof}

\begin{rem}
In the commutative case, Theorem \ref{t52} was established in \cite{ve}. It was also shown that if $\mc M = L_\infty (0, 1)$,
then for every fully symmetric Banach function space $E(0, 1)$ with the norm that is not order continuous
there exists such $T \in DS^+$ and $x \in E(\mc M)$ that the averages $M_n(x)$ do not converge in $(E(\mc M), \|\cdot\|_{E(\mc M)})$.
\end{rem}

The following proposition is a version of Theorem \ref{t51} for noncommutative fully symmetric space with order continuous norm
with condition (iii) being replaced by non-triviality of the Boyd indices of $E(0,\infty)$. Note that we do not require $T$ to be positive.

\begin{pro}
Let $E(0,\infty)$ be a fully symmetric function space with non-trivial Boyd indices and order continuous norm.
Then for any $x \in E(\mc M, \tau)$ and \\ $T \in DS(\mc M, \tau)$ there exists
such $\widehat x \in E(\mc M, \tau)$ that $\| \widehat x-M_n(x) \|_{E(\mc M)}\to 0$.
\end{pro}

\begin{proof}
 By \cite[Theorem 2.b.3]{lt}, it is possible to find such $1 < p, q < \ii$ that
$$
L_p(0,\infty) \cap L_q(0,\infty) \subset E(0,\infty) \subset L_p(0,\infty) + L_q(0,\infty)
$$
with continuous inclusion maps. In particular, $\| f \|_{E(0,\infty)} \leq C \| f \|_{L_p(0,\infty) \cap L_q(0,\infty)}$ for all
$f \in L_p(0,\infty) \cap L_q(0,\infty)$ and some $C > 0$. Hence
$$
\| x \|_{E(\mc M, \tau)} \leq C \| x \|_{L_p(\mc M, \tau) \cap L_q(\mc M, \tau)}
$$
for all $x \in \mc L:=L_p(\mc M, \tau) \cap L_q(\mc M, \tau)$. Therefore the space
$\mc L$ is continuously embedded in $E(\mc M, \tau)$.
Besides, it follows as in Theorem \ref{t52} that $\mc L$ is a fundamental subset of $(E(\mc M, \tau), \|\cdot\|_{E(\mc M)})$.

Show that for
everty $x \in \mc L$ the sequence $\{M_n(x)\}$ is relatively weakly sequentially compact in
$(E(\mc M, \tau), \|\cdot\|_{E(\mc M)})$.
Since $p,q>1$, the spaces $L_p(\mc M, \tau)$ and $L_q(\mc M, \tau)$ are reflexive.
As $T \in DS$  and $x \in L_p(\mc M, \tau) \cap L_q(\mc M, \tau)$, we conclude that
the averages $\{M_n(x)\}$ converge in $(L_p(\mc M, \tau), \|\cdot\|_p)$ and in $(L_q(\mc M, \tau), \|\cdot\|_q)$ to
$\widehat x_1 \in L_p(\mc M, \tau)$ and to $\widehat x_2 \in L_q(\mc M, \tau)$, respectively \cite[Ch.VIII, \S 5, Corollary 4]{ds}. This implies that the sequence $\{M_n(x)\}$ converges to $\widehat{x}_1$ and to $\widehat{x}_2$ in measure,
hence $\widehat x_1 = \widehat x_2 := \widehat x$. Since $\mc L$ is continuously embedded in $E(\mc M, \tau)$, the sequence $\{M_n(x)\}$ converges to $\widehat x$ with respect to the norm \ $\|\cdot\|_{E(\mc M)}$,
thus, it is relatively weakly sequentially compact in $(E(\mc M, \tau), \|\cdot\|_{E(\mc M)})$.

Now we can proceed as in the ending of the proof of Theorem \ref{t52}.
\end{proof}

\end{document}